\documentclass{monsky2009}

\usepackage{amssymb}

\usepackage{bm}

\newenvironment{conjecture*}[1][]{\textbf{Conjecture #1\hspace{.3em}}}{}
\newenvironment{theorem*}[1]{\textbf{#1}\itshape \hspace{.3em}}{\upshape}
\newenvironment{remark*}[1]{\textbf{#1}\itshape \hspace{.3em}}{\upshape}
\newenvironment{example*}[1]{\textbf{#1}\itshape \hspace{.3em}}{\upshape}
\newenvironment{proof}[1][]{\textbf{Proof #1\hspace{.3em}}}{}

\newtheorem{definition}{Definition}[section]
\newtheorem{theorem}[definition]{Theorem}
\newtheorem{lemma}[definition]{Lemma}

\newtheorem{corollary}[definition]{Corollary}

\newcounter{kpremark}
\newcommand{\mod}[1]{\ensuremath{\hspace{.5em}(#1)}}

\newcommand{\nequiv}{\ensuremath{\equiv\hspace{-1em}/\ \,}}

\newcommand{\pr}{\mathrm{pr}\ }

\begin{document}

\begin{frontmatter}






\title{A characteristic 2 recurrence related to $\bm U_{3}$, with a Hecke algebra application}
\author{Paul Monsky}

\address{Brandeis University, Waltham MA  02454-9110, USA. monsky@brandeis.edu}

\begin{abstract}
I begin with a simple modular form motivated proof of the following: Let $C_{n}$ in $Z/2[[t]]$ be defined by $C_{n+4} = C_{n+3} + (t^{4}+t^{3}+t^{2}+t)C_{n} + t^{n}(t^{2}+t)$, with initial values $0$, $1$, $t$ and $t^{2}$ for $C_{0}$, $C_{1}$, $C_{2}$ and $C_{3}$. Then every $C_{4m}$ is a sum of $C_{k}$ with $k<4m$.

This, combined with earlier results, yields: If $K$ consists of all mod $2$ modular forms of level $\Gamma_{0}(3)$ annihilated by $U_{2}$ and $U_{3} +I$, then $K$ has a basis adapted (in the sense of Nicolas and Serre) to the Hecke operators $T_{7}$ and $T_{13}$; consequently the Hecke algebra attached to $K$ is a power series ring in these two operators.

\end{abstract}


\end{frontmatter}


\section{The polynomials $\bm C_{n}$}
\label{section1}

In the finite characteristic theory of modular forms, the Hecke operators $T_{p}$ give rise to remarkable recursions. J.-L. Nicolas and J.-P. Serre \cite{3,4} used recurrences attached to $T_{3}$ and $T_{5}$ to show that the characteristic $2$ level $1$ Hecke algebra is a power series ring in these two operators. In \cite{2} I gave a level $3$ variant of the Nicolas-Serre results; my underlying space was a subquotient, $N2/N1$, of the space of mod $2$ modular forms of level $\Gamma_{0}(3)$, and the resulting Hecke algebra contained nilpotents.

In fact, the Hecke operators $U_{p}$, where $p$ divides the level, also give interesting recurrences. I begin this note with one such recurrence, coming from $U_{3}$ in characteristic $2$ and level $3$. Section \ref{section1} introduces $C_{n}$, $n\ge 0$, in $Z/2[t]$, satisfying the recurrence. Section \ref{section2} proves a basic property---each $C_{4m}$ is a $Z/2$-linear combination of $C_{k}$ with $k < 4m$.  This is built on in section \ref{section3}.  In the final section I combine my results with results from \cite{2} to show the following. Let $K$ consist of the mod $2$ modular forms of level $\Gamma_{0}(3)$ annihilated by $U_{2}$ and $U_{3}+I$. Then $K$ ``admits a basis adapted to $T_{7}$ and $T_{13}$'', and the Hecke algebra attached to $K$ is a power series ring in these two operators.

\begin{remark*}{Remark}
Sections \ref{section1}--\ref{section3} use nothing more than algebra in a polynomial ring $Z/2[r]$. The arguments I make are really modular forms arguments in transparent disguise---$Z/2[r]$ is really a ring of modular forms, $F$ and $G$ are really the mod $2$ reductions of $\Delta(z)$ and $\Delta(3z)$, $U : Z/2[r] \rightarrow Z/2[r]$ is really $U_{3}$, and $T : Z/2[F] \rightarrow Z/2[F]$ is really $T_{3}$, but all this can be ignored. The final section, however, does require some acquaintance with \cite{2} and with modular forms.
\end{remark*}


\begin{definition}
\label{def1.1}
$F$ and $G$ are the elements $r(r+1)^{3}$ and $r^{3}(r+1)$ of $Z/2[r]$. Note that $F+G = r(r+1)$ and that $F^{4}+G^{4} +FG = 0$.
\end{definition}

\begin{definition}
\label{def1.2}
$\varphi : Z/2[r] \rightarrow Z/2[r]$ is ``semi-linear'' if it is $Z/2$-linear and $\varphi(Gf)=F\varphi (f)$.
\end{definition}

Since a basis of $Z/2[r]$ as module over $Z/2[G] = Z/2[r^{4}+r^{3}]$ is ${1, r, r^{2}, r^{3}}$, a semi-linear map is determined by the images of these $4$ elements, and for any choice of $4$ elements to be images, there is a corresponding semi-linear map.

\begin{definition}
\label{def1.3}
$U : Z/2[r] \rightarrow Z/2[r]$ is the semi-linear map taking $1$, $r$, $r^{2}$ and $r^{3}$ to $1$, $r$, $r^{2}$ and $r^{3}+r^{2}+r$. Since $U(1)=1$, $U$ takes $G^{n}$ to $F^{n}$.
\end{definition}

\begin{lemma}
\label{lemma1.4}
$U(f^{2})=(U(f))^{2}$.
\end{lemma}

\begin{proof}
Since ${1, r, r^{2}, r^{3}}$ is a basis of $Z/2[r]$ over $Z/2[G]$, and $U(G)=F$, it suffices to prove this when $f$ is $1$, $r$, $r^{2}$ and $r^{3}$. The cases $f=1$ and $f=r$ are immediate. And:
\begin{itemize}
\item[] $r^{4} = G+r^{3}$. So $U(r^{4})=F+(r^{3}+r^{2}+r)=r^{4}=(U(r^{2}))^{2}$.
\item[] $r^{6} = (r^{2}+r)G+r^{4}$. So $U(r^{6})=(r^{2}+r)F+r^{4}=r^{6}+r^{2}+r^{4}=(U(r^{3}))^{2}$.\qed
\end{itemize}
\end{proof}

\begin{lemma}
\label{lemma1.5} \hspace{1em}\\
\vspace{-4.5ex}
\begin{enumerate}
\item[(a)] $U(r^{n+4})=U(r^{n+3})+(r^{4}+r^{3}+r^{2}+r)U(r^{n})$.
\item[(b)] $U((r^{2}+r)r^{2n})=(r^{2}+r)A_{n}(r^{2})$ for some $A_{n}$ in $Z/2[t]$.
\end{enumerate}
\end{lemma}

\begin{proof}
$r^{4}=r^{3}+G$. Multiplying by $r^{n}$, applying $U$ and using semi-linearity we get (a). When $n=0$, we have (b) with $A_{0}=1$. Suppose $n\ge 1$. Then by (a) and Lemma \ref{lemma1.4}, $U((r^{2}+r)r^{2n})=(r^{4}+r^{3}+r^{2}+r)U(r^{2n-2})=(r^{2}+r)(r^{2}+1)(U(r^{n-1}))^{2}$.  So if we write $U(r^{n-1})$ as $g(r)$ and set $A_{n} = (t+1)g(t)$, we get (b).\qed
\end{proof}

\begin{lemma}
\label{lemma1.6} Let $A_{n}$ be as in Lemma \ref{lemma1.5}. Then:
\begin{enumerate}
\item[(a)] $A_{n+4} = A_{n+3} + (t^{4}+t^{3}+t^{2}+t)A_{n}$.
\item[(b)] $A_{0}$, $A_{1}$, $A_{2}$ and $A_{3}$ are $1$, $t+1$, $t^{2}+t$ and $t^{3}+t^{2}$.
\end{enumerate}
\end{lemma}

\begin{proof}
$r^{8}=r^{6}+G^{2}$. Multiplying by $(r^{2}+r)r^{2n}$, applying $U$, and then dividing by $r^{2}+r$, we find that $A_{n+4}(r^{2})=A_{n+3}(r^{2})+(r^{8}+r^{6}+r^{4}+r^{2})A_{n}(r^{2})$, giving (a). $A_{0}$ is evidently $1$. Since $U$ fixes $1$, $r$ and $r^{2}$, the last sentence in the proof of Lemma \ref{lemma1.5} shows that $A_{1}$, $A_{2}$ and $A_{3}$ are $t+1$, $t^{2}+t$ and $t^{3}+t^{2}$.\qed
\end{proof}

\begin{definition}
\label{def1.7}
$C_{n}$ in $Z/2[t]$ is $A_{n}+t^{n}$, with $A_{n}$ as in Lemma \ref{lemma1.5}.
\end{definition}

\begin{theorem}
\label{theorem1.8} \hspace{1em}\\
\vspace{-5ex}
\begin{enumerate}
\item[(a)] $C_{n+4} = C_{n+3} + (t^{4}+t^{3}+t^{2}+t)C_{n}+t^{n}(t^{2}+t)$.
\item[(b)] $C_{0}$, $C_{1}$, $C_{2}$ and $C_{3}$ are $0$, $1$, $t$ and $t^{2}$.
\item[(c)] $U+I$ takes $(r^{2}+r)r^{2n}$ to $(r^{2}+r)C_{n}(r^{2})$.
\end{enumerate}

\end{theorem}

\begin{proof}
(a) and (b) are immediate from (a) and (b) of Lemma \ref{lemma1.6}, while (c) comes from Lemma \ref{lemma1.5} (b).\qed
\end{proof}

\begin{lemma}
\label{lemma1.9} 
If $C_{n}$ is a $Z/2$-linear combination of $C_{k}$ with $k<n$, then $4$ divides $n$.
\end{lemma}

\begin{proof}
Note first that $\mbox{degree}\ C_{n} = n-1$ if $n\nequiv 0\mod{4}$ and is $<n-1$ if $n\equiv 0\mod{4}$. This follows from (b) of Theorem \ref{theorem1.8} if $n<4$, and from induction using (a) of Theorem 1.8 in general. So if $n\nequiv 0\mod{4}$, $\mbox{degree}\ C_{n} = n-1$ while each $C_{k}$, $k<n$, has degree $<n-1$, and $C_{n}$ is not a $Z/2$-linear combination of such $C_{k}$.\qed
\end{proof}

\section{A key property of the $\bm C_{n}$}
\label{section2}

We shall prove a converse to Lemma \ref{lemma1.9}---each $C_{4m}$ is a $Z/2$-linear combination of $C_{k}$ with $k<4m$. In \cite{1}, I asked whether this was true, and Peter M\"{u}ller, in a computer calculation taking a few seconds, showed that it held for $4m< 10,000$. I then succeeded in finding a proof.

\begin{lemma}
\label{lemma2.1} 
$U$ takes $1$, $F$, $F^{2}$ and $F^{3}$ to $1$, $G$, $G^{2}$ and $G^{3}+F$.
\end{lemma}

\begin{proof}
$U(F+G)=U(r^{2}+r)=r^{2}+r = F+G$. Since $U(G)=F$, $U(F)=G$. So $U(F^{2})=G^{2}$. Also, Definition \ref{def1.1} shows that $F^{3}=(r+1)^{8}G$, and that $G^{3}=r^{8}F$. Since $U(r^{8})=r^{8}$, $U(F^{3})=(r+1)^{8}F=r^{8}F+F=G^{3}+F$.\qed
\end{proof}

\begin{lemma}
\label{lemma2.2} 
Let $\alpha$ be the isomorphism of $Z/2[F]$ with $Z/2[G]$ taking $F^{n}$ to $G^{n}$. Then if $P_{n}$ is either $\alpha (F^{n})=G^{n}$ or $U(F^{n})$:

\[
(\star)\qquad
P_{n+4} + F^{4}P_{n}+FP_{n+1}=0.
\]

\end{lemma}

\begin{proof}
As we saw in Definition \ref{def1.1}, $F^{4}+G^{4}+FG = 0$. Multiplying by $G^{n}$ we get ($\star$) with $P_{n}=\alpha(F^{n})=G^{n}$. Multiplying instead by $F^{n}$ and applying the semi-linear operator $U$ we get ($\star$) with $P_{n}=U(F^{n})$.\qed
\end{proof}

\begin{definition}
\label{def2.3}
For $f$ in $Z/2[F]$, $T(f)=U(f)+\alpha(f)$.
\end{definition}

\begin{theorem}
\label{theorem2.4} \hspace{1em}\\
\vspace{-5ex}
\begin{enumerate}
\item[(a)] $T$ takes $1$, $F$, $F^{2}$ and $F^{3}$ to $0$, $0$, $0$ and $F$.
\item[(b)] $T$ stabilizes $Z/2[F]$. In fact $T(F^{n})$ is a sum of $F^{k}$ with $k\le n-2$ and congruent to $n$ mod $2$.
\end{enumerate}
\end{theorem}

\begin{proof}
(a) follows from Lemma \ref{lemma2.1}. In particular the second assertion in (b) holds for $n\le 3$. Now let $P_{n}=T(F^{n})$. By Lemma \ref{lemma2.2} these $P_{n}$ satisfy the recursion ($\star$). An induction on $n$, using ($\star$), completes the proof of (b).\qed
\end{proof}

\begin{lemma}
\label{lemma2.5} 
Let $u_{0}$, $u_{1}$ and $u_{2}$ be $G$, $F$ and $(F+G)^{2}G$. Then:
\begin{enumerate}
\item[(a)] The $u_{i}$ are linearly independent over $Z/2[G]$.
\item[(b)] Each $u_{i}$ is $(r^{2}+r)g(r^{2})$ for some $g$ of degree $\le 3$.
\end{enumerate}

\end{lemma}

\begin{proof}
Since $1$, $F$ and $F^{2}$ are linearly independent over $Z/2[G]$ we get (a). The $g$ corresponding to $u_{0}$ and $u_{1}$ are $t$ and $t+1$. And since $(F+G)G=(r^{2}+r)(r^{4}+r^{3})=r^{6}+r^{4}$, the $g$ corresponding to $u_{2}$ is $t^{3}+t^{2}$.\qed

Now fix $m\ge 0$.
\end{proof}

\begin{definition}
\label{def2.6}
$L$ is the space spanned by the $u_{i}G^{2n}$ with the $u_{i}$ as above, and $0\le n\le m$. $L^{*}$ consists of the $(r^{2}+r)g(r^{2})$, where $g$ in $Z/2[t]$ has degree $\le 4m+3$.
\end{definition}

\begin{lemma}
\label{lemma2.7} 
$L$ has dimension $3m+3$ and is contained in $L^{*}$.
\end{lemma}

\begin{proof}
(a) of Lemma \ref{lemma2.5} gives the first result. Since $G^{2n} = (r^{8}+r^{6})^{n}$ and $n\le m$, (b) of Lemma \ref{lemma2.5} gives the second.\qed
\end{proof}

\begin{remark*}{Remark}
Let $M(\mathit{odd})$ consist of all elements of $Z/2[r]$ of the form $(r^{2}+r)g(r^{2})$, $g$ in $Z/2[t]$. Lemma~\ref{lemma1.5}~(b) shows that $U$ stabilizes $M(\mathit{odd})$. Moreover if the $A_{n}$ are as in Lemma \ref{lemma1.5}, then Lemma \ref{lemma1.6} and an induction show that $\mbox{degree}\ (A_{n})\le n$. It follows that $U$ stabilizes $L^{*}$. (But when $m\ge 1$, $U$ does not stabilize $L$. For $G^{3}=G^{2}u_{0}$ lies in $L$, but $U(G^{3})=F^{3}$ is not even a $Z/2[G]$-linear combination of $u_{0}$, $u_{1}$ and $u_{2}$).
\end{remark*}

\begin{lemma}
\label{lemma2.8} 
For $0\le i\le2$, $(U+I)^{2} = U^{2}+I$ maps $F^{i}G^{k}$ to $F^{i}T(F^{k})$.
\end{lemma}

\begin{proof}
$U(F^{i}G^{k})=F^{k}U(F^{i})$; since $i\le 2$ this is $F^{k}G^{i}$. Then $U^{2}(F^{i}G^{k})=F^{i}(U(F^{k}))=F^{i}G^{k}+F^{i}T(F^{k})$, and $U^{2}+I$ takes $F^{i}G^{k}$ to $F^{i}T(F^{k})$.\qed
\end{proof}

\begin{theorem}
\label{theorem2.9}
Let $K_{m}$ be the kernel of $U+I:L^{*}\rightarrow L^{*}$. (The remark shows that $U+I$ stabilizes $L^{*}$.) Then the dimension of $K_{m}$ is greater than or equal to $m+1$.
\end{theorem}

\begin{proof}
Each $u_{i}G^{2n}$ with $n\le m$ is a sum of $F^{i}G^{k}$ where $i\le 2$ and $i+k$ is both $\le 2m+3$ and odd. By Lemma \ref{lemma2.8}, the image of such an element under $(U+I)^{2}$ is a sum of $F^{i}T(F^{k})$ with $i+k\le 2m+3$ and odd.  By Theorem \ref{theorem2.4}~(b) each such $F^{i}T(F^{k})$ is in the space spanned by the $F^{n}$ with $n\le 2m+1$ and odd. It follows that the image of $L$ under $(U+I)^{2}$ has dimension $\le m+1$, and that the kernel has dimension $\ge (3m+3)-(m+1)=2m+2$. Since $L\subset L^{*}$, $(U+I)^{2}:L^{*}\rightarrow L^{*}$ has a kernel of dimension at least $2m+2$, and the dimension of $K_{m}$ is at least $m+1$.\qed
\end{proof}

\begin{theorem}
\label{theorem2.10}
$C_{4m}$ is a $Z/2$-linear combination of $C_{k}$ with $k<4m$.
\end{theorem}

\begin{proof}
Let $L$ and $L^{*}$ be as in Definition \ref{def2.6}, and $K_{m}$ be as in Theorem \ref{theorem2.9}. Suppose $f=(r^{2}+r)g(r^{2})$ is in $K_{m}$. Write $g$ as $t^{j}+$ (a sum of $t^{k}$ with $k<j$). Applying $U+I$ and using Theorem \ref{theorem1.8}~(c), we find that $C_{j}$ is the sum of the corresponding $C_{k}$. So by Lemma \ref{lemma1.9}, $4$ divides $j$. Since $j\le 4m+3$, the degree, $2j+2$, of $f$ in $r$ is an element of $[0, 8m+8]$ that is congruent to $2\bmod 8$. Note that there are exactly $m+1$ such elements. Now $K_{m}$ admits a $Z/2$-basis consisting of elements having distinct degrees in $r$. We've just shown that the only possible degrees are the $m+1$ integers in $\{2, 10, \cdots, 8m+2\}$. Since $K_{m}$ has dimension at least $m+1$, all of these degrees do occur, and there is an $f=(r^{2}+r)g(r^{2})$ in $K_{m}$ with $\mbox{degree}\ g=4m$. Write $g$ as $t^{4m}+$ a sum of $t^{k}$ with $k<4m$. Applying $U+I$ and using Theorem \ref{theorem1.8}~(c) once again we find that $C_{4m}$ is the sum of the corresponding $C_{k}$.\qed
\end{proof}

We conclude this section with:

\begin{lemma}
\label{lemma2.11} 
The kernels of $(U+I)^{2}$ acting on $L$ and on $L^{*}$ are the same.
\end{lemma}

\begin{proof}
The proof of Theorem \ref{theorem2.9} shows that the first kernel has dimension at least $2m+2$. So it's enough to show that the second kernel has dimension $\le 2m+2$. But in the proof of Theorem~\ref{theorem2.10} we constructed a basis of $K_{m}$ having $m+1$ elements.\qed
\end{proof}

\section{The structure of  $\bm K$}
\label{section3}
In the last section we constructed a sequence of subspaces $K_{0}\subset K_{1}\subset K_{2}\subset \ldots$ of $Z/2[r]$ and, in the proof of Theorem \ref{theorem2.10}, showed that $K_{m}$ has dimension $m+1$.

\begin{definition}
\label{def3.1}
$K$ is the union of the $K_{m}$. Alternatively, $K$ is the kernel of $U+I:M(\mathit{odd})\rightarrow M(\mathit{odd})$, where $M(\mathit{odd})$ consists of the $(r^{2}+r)g(r^{2})$, $g$ in $Z/2[t]$.
\end{definition}

In this section we construct an isomorphism between $K$ and a certain subquotient, $K1/N1$, of $Z/2[r]$.

\begin{definition}
\label{def3.2}
$N2$, $K1$, $K5$, and $N1$ are the (free) $Z/2[G^{2}]$-submodules of $Z/2[r]$ with the following bases.

\begin{tabular}{ll}
N2: & A basis is $\{G,F,F^{2}G\}$.\\[-1ex]
K1 and K5: & Bases are $\{G,F\}$ and $\{G,F^{2}G\}$.\\[-1ex]
N1: & A basis is $\{G\}$.
\end{tabular}

Note that $N2/N1$ is the direct sum of $K1/N1$ and $K5/N1$. This gives us a projection map, $\pr 1:N2/N1\rightarrow K1/N1$.
\end{definition}

\begin{lemma}
\label{lemma3.3} 
$K\subset N2$.
\end{lemma}

\begin{proof}
Suppose $f$ is in $K$. Then $f$ is in some $K_{m}$, $f$ is in an $L^{*}$, and by Lemma \ref{lemma2.11}, $f$ is in an $L$. So $f$ is a $Z/2[G^{2}]$-linear combination of $u_{0}=G$, $u_{1}=F$ and $u_{2}=F^{2}G+G^{3}$.\qed
\end{proof}

Composing the inclusion of $K$ in $N2$ with $\pr 1$, we get a map $K\rightarrow K1/N1$ which we also call $\pr 1$. Our goal is to show that this $\pr 1$ is bijective.

\begin{lemma}
\label{lemma3.4} 
Suppose $f\ne 0$ is in $K$. Write $f$ as $(r^{2}+r)g(r^{2})$ with degree $g$ equal to $4m$. Then $\pr 1\,(f) = (G^{2m}+ \mbox{ a sum of }G^{2k}, k<m)\cdot F$.

\end{lemma}

\begin{proof}
$G^{2n}(F+G)$, $G^{2n}(G)$ and $G^{2n+1}(F+G)^{2}$ have degrees $8n+2$, $8n+4$ and $8n+8$ in $r$. Lemma \ref{lemma3.3} shows that $f$, which has degree $8m+2$, is a sum of $G^{2m}(F+G)$, various $G^{2k}(F+G)$ with $k<m$, various $G^{2n}(G)$ and various $G^{2n}(F^{2}G+G^{3})$. Applying $\pr 1$ which annihilates each $G^{2n}(G)$ and each $G^{2n}(F^{2}G)$, we get the result.

It's now easy to see that $\pr 1 : K\rightarrow K1/N1$ is bijective. Recall that for each $m\ge 0$ there is an element $(r^{2}+r)g(r^{2})$ of $K$ with degree $g=4m$, and that furthermore if we choose one such element for each $m\ge 0$ we get a  $Z/2$-basis of $K$. Since the $G^{2m}F$, $m\ge 0$, form a $Z/2$-basis of $K1/N1$, Lemma \ref{lemma3.4} gives the result.\qed
\end{proof}

\section{The action of  $\bm T_{p}$ and the main theorem}
\label{section4}

Up to now, everything we've done is algebra in a polynomial ring $Z/2[r]$. We now connect this with modular forms, using language from \cite{2}. Instead of $r$ being an indeterminate it is now the explicit element $\sum_{n>0}(x^{n^{2}}+x^{2n^{2}}+x^{3n^{2}}+x^{6n^{2}})$ of $Z/2[[x]]$. Then $Z/2[r]$ is a subspace of $Z/2[[x]]$ and \cite{2}, Corollary 2.5, shows that this subspace is the space, $M$, of mod $2$ modular forms of level $\Gamma_{0}(3)$. (See \cite{2}, Definitions 2.2 and 2.3.)

\begin{definition}
\label{def4.1}
$U_{2}$, $U_{3}$ and $T_{p}$ ($p$ a prime $>3$) are the following commuting operators on $Z/2[[x]]$.
\vspace{-2ex}
\begin{eqnarray*}
U_{2}(\sum a_{n}x^{n}) &=& \sum a_{2n}x^{n}.\\
U_{3}(\sum a_{n}x^{n}) &=& \sum a_{3n}x^{n}.\\
T_{p}(\sum a_{n}x^{n}) &=& \sum a_{pn}x^{n} + \sum a_{n}x^{pn}.
\end{eqnarray*}
\end{definition}

\begin{lemma}
\label{lemma4.2} 
The above operators stabilize $Z/2[r]$.
\end{lemma}

This is immediate from the modular forms interpretation of $Z/2[r]$.

\begin{lemma}
\label{lemma4.3} \hspace{1em}\\
\vspace{-5ex}
\begin{eqnarray*}
F&=&\textstyle \sum_{n\ \mathrm{odd},\ n>0} x^{n^{2}}.\\
G&=&F(x^{3}).
\end{eqnarray*}
\end{lemma}

\begin{proof}
See \cite{2}, Theorem 2.8, and the first sentence of \cite{2}, section 2.\qed
\end{proof}

\begin{lemma}
\label{lemma4.4} 
$U_{3} : Z/2[r]\rightarrow Z/2[r]$ is the map $U$.
\end{lemma}

\begin{proof}
Lemma \ref{lemma4.3} shows that $U_{3}(Gf) = FU_{3}(f)$. So $U_{3}$, like $U$, is semi-linear, and it's enough to show that they agree on a basis of $Z/2[r]$ over $Z/2[G]$. But they each fix $1$, $r$, $r^{2}$ and $r^{4}$.\qed
\end{proof}

\begin{lemma}
\label{lemma4.5} 
The $M(\mathit{odd})$ of Definition \ref{def3.1} consists of those elements of $Z/2[r]$ that are odd power series in $x$. In other words, it is the kernel of $U_{2}: Z/2[r]\rightarrow Z/2[r]$.
\end{lemma}

\begin{proof}
See \cite{2}, Definition 2.6 and Theorem 2.7.\qed
\end{proof}

Combining Lemmas \ref{lemma4.4} and \ref{lemma4.5} we find:

\begin{theorem}
\label{theorem4.6}
The kernel $K$ of $U+I: M(\mathit{odd})\rightarrow M(\mathit{odd})$ is just the space consisting of those mod $2$ modular forms of level $\Gamma_{0}(3)$ that are annihilated by $U_{2}$ and by $U_{3}+I$. In particular the $T_{p}$, $p>3$, stabilize $K$.
\end{theorem}

\begin{definition}
\label{def4.7}
For $i=1$ or $2$, $p_{3,i}: Z/2[[x]]\rightarrow Z/2[[x]]$ is the map $\sum a_{n}x^{n}\rightarrow \sum_{n\,\equiv\, i\mod{3}} a_{n}x^{n}$.
\end{definition}

\begin{lemma}
\label{lemma4.8} 
$K1$ consists of those elements of $M(\mathit{odd})$ annihilated by $p_{3,2}$, $K5$ of those elements annihilated by $p_{3,1}$.
\end{lemma}

\begin{proof}
See \cite{2}, Definition 2.24 and Theorem 2.15.\qed
\end{proof}

\begin{corollary}
\label{corollary4.9}
The $T_{p}$, $p\equiv 1\mod{6}$, stabilize $K1$ and $K5$. The $T_{p}$ with $p\equiv 5\mod{6}$ map $K1$ into $K5$ and $K5$ into $K1$. The $T_{p}$ with $p>3$ stabilize both $N2=K1+K5$ and $N1=K1\cap K5$.
\end{corollary}

Now the $T_{p}$ with $p\equiv 1\mod{6}$ act on $K$; by Corollary \ref{corollary4.9} they also act on $K1/N1$.

\begin{lemma}
\label{lemma4.10} 
The bijection $\pr 1 : K\rightarrow K1/N1$ of the last section preserves the action of the $T_{p}$, $p\equiv 1\mod{6}$.
\end{lemma}

\begin{proof}
$\pr 1$ is the composition of $K\subset N2$, $N2\rightarrow N2/N1$ and the projection map $N2/N1\rightarrow K1/N1$ of the sentence preceding Lemma \ref{lemma3.3}. The first two of these maps preserve the action of 
$T_{p}$, $p>3$, while the third preserves the action of $T_{p}$, $p\equiv 1\mod{6}$. (Here we use Lemma \ref{lemma4.8}.)\qed
\end{proof}

\begin{definition}
\label{def4.11}
$D=x+x^{25}+x^{49}+\cdots$ is $p_{3,1}(F)$. $W1$ is spanned by the $D^{k}$, $k\equiv 1\mod{6}$, $W5$ by the $D^{k}$, $k\equiv 5\mod{6}$.
\end{definition}

Now let $X$ and $Y$ be the Hecke operators $T_{7}$ and $T_{13}$. These act on a number of spaces we've considered, and in particular on $W1$, $W5$, $K1/N1$, $K5/N1$ and $K$. 

\begin{lemma}
\label{lemma4.12}
$W5$ has ``a basis adapted to $X$ and $Y$'' with $m_{0,0} = D^{5}$. More precisely there is a $Z/2$-basis $m_{i,j}$ ($i\ge 0, j\ge 0$) of $W5$ such that:
\begin{enumerate}
\item[(1)] $m_{0,0} = D^{5}$.
\item[(2)] $Xm_{i,j} = m_{i-1,j}$ or $0$ according as $i>0$ or $i=0$.
\item[(3)] $Ym_{i,j} = m_{i,j-1}$ or $0$ according as $j>0$ or $j=0$.
\end{enumerate}

Similarly, $W1$ has a basis adapted to $X$ and $Y$ with $m_{0,0}=D$.
\end{lemma}

\begin{proof}
The first result is \cite{2}, Corollary 4.13. And \cite{2}, Theorem 4.20 shows that $T_{5}:W5\rightarrow W1$ is bijective, takes $D^{5}$ to $D$ and preserves the action of $T_{p}$, $p\equiv 1\mod{6}$. So we deduce the second result from the first.\qed
\end{proof}

\begin{theorem}
\label{theorem4.13} \hspace{1em}\\
\vspace{-5ex}
\begin{enumerate}
\item[(a)] $K1/N1$ has a basis adapted to $X$ and $Y$ with $m_{0,0}=F$.
\item[(b)] $K5/N1$ has a basis adapted to $X$ and $Y$ with $m_{0,0}=F^{2}G$.
\end{enumerate}
\end{theorem}

\begin{proof}
In the remark following the proof of \cite{2}, Theorem 2.17, we constructed an isomorphism of $N2/N1$ with $W=W1\oplus W5$ which commutes with the $T_{p}$, $p>3$. This isomorphism takes $F$ to $D$, $F^{2}G$ to $D^{5}$, $K1/N1$ to $W1$ and $K5/N1$ to $W5$. So the results follow from Lemma \ref{lemma4.12}.
\qed
\end{proof}

We now prove the main theorem---$K$ (which by Theorem \ref{theorem4.6} consists of those mod $2$ modular forms of level $\Gamma_{0}(3)$ annihilated by $U_{2}$ and $U_{3}+I$) has a basis adapted to $T_{7}$ and $T_{13}$ with $m_{0,0}=F+G$, and as a consequence the Hecke algebra attached to $K$ is a power series ring in $T_{7}$ and $T_{13}$.

Note first that $\pr 1$ takes the element $F+G$ of $K$ to $F$. Then Theorem \ref{theorem4.13}, combined with Lemma \ref{lemma4.10}, shows that $K$ has a basis adapted to $T_{7}$ and $T_{13}$ with $m_{0,0}=F+G$. It also follows that $K$ has the structure of $Z/2[[X,Y]]$-module with $X$ and $Y$ acting by $T_{7}$ and $T_{13}$, and that this action is faithful. Now each $T_{p}$, $p>3$, stabilizes $K$. It remains to show that $T_{p}:K\rightarrow K$ is multiplication by some element of $(X,Y)$. Now $T_{p}$ commutes with $T_{7}$ and $T_{13}$ and so is $Z/2[[X,Y]]$-linear. The existence of an adapted basis for $K$ then shows (see the proof of \cite{2}, Theorem 4.16) that $T_{p}:K\rightarrow K$ is multiplication by some $u$ in $Z/2[[X,Y]]$. Since $T_{p}(m_{0,0})=T_{p}(F+G)=0$, $u$ is in $(X,Y)$, and we're done.

We conclude this note by comparing two Hecke algebras. The inclusion of $K$ in $N2$ gives a map $K\rightarrow N2/N1$ which is 1--1. The $T_{p}$, $p>3$, act both on $K$ and $N2/N1$ and we have corresponding Hecke algebras $\mathrm{HE}(K)$ and $\mathrm{HE}(N2/N1)$. The inclusion of $K$ in $N2/N1$ gives an onto ring homomorphism $\mathrm{HE}(N2/N1)\rightarrow \mathrm{HE}(K)$ taking $T_{p}$ to $T_{p}$ for all $p>3$. Now in \cite{2} we've shown that $\mathrm{HE}(N2/N1)$ is a power series ring in $T_{7}$ and $T_{13}$ with an element $\varepsilon$ of square $0$ adjoined, and that $\varepsilon$ can be taken to be $T_{5}+\lambda(T_{7},T_{13})$ for a certain two-variable power series, $\lambda$. Since $\mathrm{HE}(K)$ is, as we've shown, a two-variable power series ring, the kernel of $\mathrm{HE}(N2/N1)\rightarrow \mathrm{HE}(K)$ is a height $0$ prime ideal of $\mathrm{HE}(N2/N1)$ which can only be $(\varepsilon)$. Let $\mathrm{HE}(N2/N1)_{\mathrm{red}}$ be the quotient of $\mathrm{HE}(N2/N1)$ by its nilradical. We've shown:

\begin{theorem}
\label{theorem4.14}
There is an isomorphism of $\mathrm{HE}(N2/N1)_{\mathrm{red}}$ with $\mathrm{HE}(K)$ taking $T_{p}$ to $T_{p}$ for each $p>3$.
\end{theorem}

\begin{remark*}{Remark}
Since $\varepsilon$ annihilates $K$ it annihilates the image of $K$ in $N2/N1$. So this image $\subset \varepsilon\cdot(N2/N1)$. It's easy to see that the image is all of $\varepsilon\cdot(N2/N1)$.
\end{remark*}



\end{document}